\documentclass{amsart}
\usepackage{amsmath}
\usepackage{graphicx}
\usepackage{amsfonts}
\usepackage{amssymb}

\newtheorem{theorem}{Theorem}
\theoremstyle{plain}

\newtheorem{corollary}{Corollary}

\newtheorem{definition}{Definition}
\newtheorem{example}{Example}

\newtheorem{proposition}{Proposition}
\newtheorem{remark}{Remark}

\numberwithin{equation}{section}

\begin{document}

\author{Manuel Rivera and Samson Saneblidze}

\newcommand{\Addresses}{{
  \bigskip
  \footnotesize

   \textsc{Manuel Rivera, Department of Mathematics, University of Miami, 1365 Memorial Drive, Coral
Gables, FL 33146 and
Departamento de Matem\'aticas, Cinvestav, Av. Instituto Polit\'ecnico Nacional 2508, Col. San Pedro Zacatenco, M\'exico, D.F. CP 07360, M\'exico} \par\nopagebreak
  \textit{E-mail address} \texttt{manuelr@math.miami.edu}

  \medskip
  \medskip

  \textsc{Samson Saniblidze,  A. Razmadze Mathematical Institute
I.Javakhishvili Tbilisi State University
6, Tamarashvili  st.,
 Tbilisi 0177, Georgia}
 \par\nopagebreak
  \textit{E-mail address} \texttt{sane@rmi.ge }

}}

\begin{abstract}
We introduce the abstract notion of a necklical set in order to describe a functorial combinatorial model of the path fibration over the geometric realization of a path connected simplicial set. In particular, to any path connected simplicial set $X$ we associate a necklical set $\widehat{\mathbf{\Omega}}X$ such that its geometric realization $|\widehat{\mathbf{\Omega}}X|$, a space built out of gluing cubical cells, is homotopy equivalent to the based loop space on $|X|$ and the differential graded module of chains $C_*(\widehat{\mathbf{\Omega}}X)$ is a differential graded associative algebra generalizing Adams' cobar construction.
\end{abstract}

\title[A combinatorial model for the path fibration]{A combinatorial model for the path fibration}
\maketitle

 \section{Introduction}

 The cobar construction, as introduced by Adams in his celebrated paper \cite{Adams}, describes a functorial way of producing a differential graded associative algebra (dg algebra) from a connected differential graded coassociative coalgebra (dg coalgebra). Adams proved that such a construction models the passage that starts with the dg coalgebra of chains on a simply connected topological space $Y$ and goes to the dg algebra of chains on the based (Moore) loop space $\Omega Y$ of $Y$. More precisely, if $(\operatorname{Sing}^1_n(Y,y), \partial, \Delta)_{n\geq 0},$
    denotes the connected dg coalgebra (over a fixed ring $\Bbbk$) of singular chains in $Y$ with edges collapsed to a fixed point $y \in Y$, then the cobar construction on the quotient dg coalgebra  $\operatorname{Sing}^1(Y):=(\operatorname{Sing}^1_*(Y,y) /  \operatorname{Sing}^1_{*>0}(y) , \partial, \Delta)$
    is a dg algebra quasi-isomorphic to the normalized singular cubical chains of $\Omega Y.$
      A motivation of the cobar construction was the Adams-Hilton model of $\Omega Y$ \cite{Adams- Hilton}, which in turn generalizes the James model of $\Omega \Sigma X $. In all models the simply connected hypothesis was assumed
 in order to apply a spectral sequence comparison theorem to prove
the statements.

S. Saneblidze and T. Kadeishvili showed in \cite{Saneblidze- Kadeishvili} that the cobar construction on
        $\operatorname{Sing}^1(Y)$
        is isomorphic to the chain complex associated to a cubical set model (without degeneracies)  of $\Omega Y$. On the other hand, M. Rivera and M. Zeinalian proved in \cite{Rivera- Zeinalian} that for any connected space $Y$ the cobar construction of $\operatorname{Sing}^N(Y,y)$, the dg coalgebra of the normalized singular chains with vertices at $y$, yields a dg algebra quasi-isomorphic to the singular chains of $\Omega Y$ and, moreover, such a dg algebra is isomorphic to the normalized chains associated to certain cubical set with ``connections", a notion introduced in \cite{Brown- Higgins}. This statement is also true if we take the cobar construction on the dg coalgebra of chains associated to any Kan complex model of $Y$. The proof relied on some basic results from the theory of $\infty$-categories and on classifying morphisms in the category of \textit{necklaces}, i.e. simplicial sets of the form $\Delta^{n_1} \vee ... \vee \Delta^{n_k}$ where $\Delta^{n_i}$ is a standard simplex of dimension $n_i$ and each wedge $\Delta^{n_i} \vee \Delta^{n_i+1}$ is obtained by identifying the last vertex of $\Delta^{n_i}$ with the first vertex of $\Delta^{n_i+1}$. Necklaces were introduced by D. Dugger and D. Spivak in \cite{Dugger- Spivak} to describe the mapping spaces of the \textit{rigidification} functor from simplicial sets to simplicial categories, the right adjoint of the homotopy coherent nerve functor. In \cite{Rivera- Zeinalian} the relationship between the mapping spaces of the rigidification functor, the cobar construction of the dg coalgebra of normalized chains of a Kan complex model for a space, and the based loop space is studied in detail.

In this article we use
 construct a quasi-fibration $|\widehat{\mathbf{\Omega}}X| \xrightarrow{\iota} |\widehat{\mathbf{P}} X| \xrightarrow{\xi} Y$ modeling the path fibration
 $\Omega Y \rightarrow PY \rightarrow Y$, where $Y=|X|$ is the geometric realization of a path connected simplicial set $X,$
 while $\widehat{\mathbf{P}} X$ and $\widehat{\mathbf{\Omega }} X$
  are abstract  sets whose modeling polytopes are the standard cubes indexed by necklaces inside $X$. Since the restriction of $\xi$ to $|\widehat{\mathbf{P}} X^1|\rightarrow |X^1|$  is a covering ($X^1$ is the $1$-skeleton of $X$), this model can be thought of as an extension of the Cayley covering on the wedge of circles.

 More precisely, $\widehat{\mathbf{\Omega}}X$ and  $\widehat{\mathbf{P}} X$ are \textit{necklical sets},
  i.e. presheaves over certain categories of necklaces. Necklical sets may be thought of as cubical sets equipped with a particular set of degeneracies making them lie somewhere in between classical cubical sets and cubical sets with connections. We prove using basic tools from classical algebraic topology that the geometric realization of the necklical set $\widehat{\mathbf{\Omega}}X$ is a topological space homotopy equivalent to the based loop space on $|X|$. The construction of the necklical set $\widehat{\mathbf{\Omega}}X$ involves a strict localization step described in section 3.4 as opposed to the Kan replacement step suggested by the constructions in \cite{Rivera- Zeinalian}. The result is a model for the based loop space of $|X|$ which is smaller than the one described in \cite{Rivera- Zeinalian} and therefore suitable for calculations.

 Moreover, we show that if we take the normalized chains associated to $\widehat{\mathbf{\Omega}}X$ we obtain a dg algebra generalizing Adams' cobar construction on the dg coalgebra of chains on $X$ as well as the extended cobar construction of K. Hess and A. Tonks as described in \cite{Hess- Tonks} when the simplicial set $X$ has a single vertex. The methods in the proof of the main result (Theorem 1) also go through to give a proof of the result in \cite{Rivera- Zeinalian} stated above in the second paragraph without relying on the theory of $\infty$-categories as explained in Remarks 2 and 3.

\section{Necklaces and necklical sets}

Denote by $Set_{\Delta}$ the category of simplicial sets and by $\Delta^{m}\in Set_{\Delta}$ the standard $m$-simplex. A \textit{necklace} is a wedge of standard simplices $T=\Delta^{n_1} \vee ... \vee \Delta^{n_k} \in Set_{\Delta}$ where the last vertex of $\Delta^{n_i}$ is identified with the first vertex of $\Delta^{n_{i+1}}$ and $n_i \geq 1$ for $i=1,...,k$. Each $\Delta^{n_i}$ is a subsimplicial set of $T$, which we call a \textit{bead} of $T$. The number of beads is denoted by $b(T).$ The set $T_0$, or the \textit{vertices} of $T$, inherits an ordering from the ordering of the beads in $T$ and the ordering of the vertices of each $\Delta^{n_i}$. A morphism of necklaces $f: T \to T'$  is a morphism of simplicial sets which preserves first and last vertices.
Denote by $Nec$ the category of necklaces.

We define a new category $Nec_0$ whose objects are necklaces with at least two beads and whose first bead is allowed to be of dimension $0$, these are called \textit{augmented necklaces}, and morphisms are maps of necklaces which preserve the first bead and last vertex. More precisely, objects in $Nec_0$ are simplicial sets of the form $\Delta^n \vee T$ where $n \geq 0$, $T$ is a necklace in $Nec$, and the last vertex of $\Delta^n$ is identified with the first vertex of $T$. A morphism between objects $S= \Delta^n \vee T$ and $S'=\Delta^{n'} \vee T'$  in $Nec_0$ is a map $g: S \to S'$ of simplicial sets sending the last vertex of $S$ to the last vertex of $S'$ and mapping the first bead $\Delta^n$ of $S$ into the first bead of $\Delta^{n'}$ of $S'$, in other words, satisfying $g(\Delta^n) \subseteq \Delta^{n'}$.

For any simplicial set  $X$ and vertices $x,y \in X_0$ define $(Nec \downarrow X)_{x,y}$ to be the category whose objects are maps of simplicial sets $(f: T \to X)$ where $T \in Nec$ and $f$ sends the first vertex of $T$ to $x$ and the last vertex of $T$ to $y$ and a morphism between objects $(f: T \to X)$ and $(f': T' \to X)$ in $(Nec \downarrow X)_{x,y}$ is given by a morphism $u: T \to T'$ in $Nec$ satisfying $f= f' \circ u$. Define $(Nec_0 \downarrow X)_{y}$ similarly but now objects are maps of simplicial sets $f: T \to X$ where $T \in Nec_0$ and $f$ sends the last vertex of $T$ to $y$.

A \textit{necklical set} is a functor $K: Nec^{op} \rightarrow Set$ and a morphism of necklical sets is given by a natural transformation of functors. Denote the category of necklical sets by $Set_{Nec}$.
A simplicial set $X$ with fixed vertices $x$ and $y,$  gives rise to an  example of a necklical set
$K_X  : Nec^{op} \rightarrow Set$
 via  the assignment
$K_X(T)= Hom(T,X)_{x,y},$ the set of simplicial maps that send the first vertex of $T$ to $x$ and the last vertex of $T$ to $y.$ Similarly, given a necklace $T \in Nec$ we will denote by $Y(T)$ the necklical set obtained through the Yoneda embedding, namely, $Y(T): Nec^{op} \to Set$ is defined by $Y(T):= \text{Hom}_{ Nec }(\_ \, , T)$.  For any two necklical sets $P$ and $Q$ define their product $P \times Q$ to be the necklical set

\begin{eqnarray*}
 \underset{\substack{ Y(T) \rightarrow P, Y(S) \rightarrow \,Q \\ T,S \in Nec}} {\text{colim}}\,\, Y(T \vee S).
 \end{eqnarray*}

In a similar way we define \textit{augmented necklical sets} as functors $L: Nec_0^{op} \to Set$ and denote the category of augmented necklical sets by $Set_{Nec_0}$.
 A simplicial set $X$ with a fixed vertex $y,$  gives rise an  example of an augmented  necklical set
 $L_X: Nec_0^{op} \to Set$
 via  the assignment
$L_X(T)=Hom(T,X)_{y},$ the set of simplicial maps that send  the last vertex of $T$ to $y.$

The \textit{dimension} of a necklace $T = \Delta^{n_1} \vee ... \vee \Delta^{n_k} \in Nec$ is defined by $\text{dim}(T)=n_1 + ... + n_k -k$, while the dimension of an augmented necklace $S=\Delta^{n_0} \vee \Delta^{n_1} \vee ... \vee \Delta^{n_k} \in Nec_0$ is defined by $\text{dim}(T)=n_0 + n_1 + ... + n_k -k$. Let
 \[Nec^n=\{T\in Nec \mid \dim T=n\}\]
 and
 \[Nec^{n,k}=\{T\in Nec^n \mid  b(T)=k\}.\]
 Given $K \in Set_{Nec}$ and an integer $n\geq 0$, define $K_n$ to be the set
\begin{eqnarray*}
K_n:= \bigsqcup_{ T \in Nec^n}  K(T)= \bigsqcup_{T \in Nec^n }  \text{Hom}_{Set_{Nec}}(Y(T), K),
\end{eqnarray*}
and
\begin{eqnarray*}
K_{n, k}:= \bigsqcup_{ T \in Nec^{n,k}}  K(T).
\end{eqnarray*}
For any $L \in Set_{Nec_0}$, the sets $L_n$  and $L_{n,k}$ are    defined similarly.

Morphisms of necklaces are generated by three types of morphisms described in the proposition below. For a proof see \cite{Rivera- Zeinalian} (Proposition 3.1).

\begin{proposition} Any non-identity morphism in $Nec$ is a composition of morphisms of the following type

\begin{itemize}
\item [(i)]  $f: T \to T'$ is an injective morphism of necklaces and $ \dim(T') - \dim(T) =1;$

\item [(ii)] $f: \Delta^{n_1} \vee ... \vee \Delta^{n_k} \to \Delta^{m_1} \vee ... \vee \Delta^{m_k}$ is a morphism of necklaces of the form $f=f_1 \vee ... \vee f_k$ such that for exactly one $p$
    with $n_p\geq 2,$
     $f_p=s^j: \Delta^{n_p} \to \Delta^{m_p}$ is a codegeneracy morphism $s^j$ for some $j$ (so $m_p=n_p-1$) and for all $i \neq p$, $f_i: \Delta^{n_i}  \to \Delta^{m_i}$ is the identity map of standard simplices (so $n_i=m_i$ for $i \neq p$);

\item [(iii)] $f: \Delta^{n_1} \vee ...\vee \Delta^{n_{p-1}} \vee \Delta^{n_p} \vee \Delta^{n_{p+1}} \vee... \vee  \Delta^{n_k} \to \Delta^{n_1} \vee ...\vee \Delta^{n_{p-1}} \vee \Delta^{n_{p+1}} \vee... \vee  \Delta^{n_k}$ is a morphism of necklaces such that,
    for some $1 \leq p \leq k$ with $k\geq 2,$  $f$ collapses the $p$-th bead $\Delta^{n_p}$ in the domain to the last vertex of  $\Delta^{n_{p-1}}$ (or to the first vertex of  $\Delta^{n_{p+1}}$)
    in the target
    and the restriction of $f$ to all the other beads is injective.
\end{itemize}
\end{proposition}

The morphisms of type $(i)$ and $(ii)$ can be furthered classified. For a morphism $f: T= \Delta^{n_1} \vee ... \vee \Delta^{n_k} \to T'=\Delta^{m_1} \vee... \vee \Delta^{m_l}$ of type $(i)$ we have two sub-types: $(ia)$ the number of vertices of $T$ is one less than the number of vertices of $T'$ (in particular, this implies $k=l$) and $(ib)$ the number of vertices of $T$ and $T'$ are equal (which, in particular, implies $l=k+1$). Morphisms of type $(ia)$ are of the form $f= id \vee d^{j} \vee id$ where $d^j: \Delta^{n_i} \to \Delta^{m_i}$, for some $i$ with $n_i+1=m_i$, is the simplicial $j$-th co-face morphism for some $j \in \{1,...,n_{i}-1\}$. Morphisms $f$ of type $(ib)$ are those for which there are $i \in \{1,...,k \}$ and $j \in \{1,...,m_i-1\}$ such that $f=id \vee W^j \vee id$, where $W^j: \Delta^{n_i} \vee \Delta^{n_{i+1}} \to  \Delta^{m_i}$ for $n_i + n_{i+1}=m_i$ and $W^j$ is the injective map whose image in $\Delta^{m_i}$ is the wedge of the two sub-simplicial sets corresponding to the $j$-th term in the Alexander-Whitney diagonal map applied to the unique non-degenerate top dimensional simplex in $\Delta^{m_i}$.  Given a necklace $T$ of dimension $n$ there are exactly $n$ morphisms $d^{i,T}_1: T^1_i \to T$ $(i=1,...,n)$ of type $(ia)$ and $n$ morphisms $d_0^{i,T}: T^0_i \to T$ $(i=1,...,n)$ of type $(ib)$. Morphisms of type $(ii)$ can be classified into two types as well: $(iia)$ those for which  $f_p=s^0$ or $f_p=s^{n_p}$ (i.e. $f_p$ is the first or last co-degeneracy morphism) and $(iib)$ otherwise. We will sometimes abuse notation and write $f=s^j$ for morphisms of type $(ii)$, omitting the index $p$ in the notation which indicates the bead to which the co-degeneracy $s^j$ is applied. A similar classification result holds for morphisms in $Nec_0$.

 \section{Necklical models for the based loop space and the based path space}

 For any simplicial set $X$ and $x,y \in X_0$ consider the graded set \[\bigsqcup_{T\in Nec} Hom(T,X)_{x,y}/\sim\]
 where $\sim$ is the equivalence relation generated by the following two rules:
\begin{eqnarray}\label{rule1}
f \circ s^{n_p+1} \sim f\circ  s^{0},  \, 1\leq p< k,
\end{eqnarray}
for any $f: \Delta^{n_1} \vee ... \vee \Delta^{n_k} \rightarrow  X  \in   Hom(T,X)_{x,y}$, where
\begin{multline*}
s^{n_p+1}:\Delta^{n_1} \vee ...\vee \Delta^{n_{p-1}} \vee \Delta^{n_p+1} \vee  \Delta^{n_{p+1}}\vee... \vee  \Delta^{n_k} \to \\
\Delta^{n_1} \vee ... \vee \Delta^{n_{p-1}} \vee \Delta^{n_p} \vee \Delta^{n_{p+1}} \vee ...\vee \Delta^{n_k}
\end{multline*} is given by applying the last co-degeneracy map to the $p$-th bead, and
\begin{multline*}
s_0: \Delta^{n_1} \vee ...\vee \Delta^{n_p} \vee \Delta^{n_{p+1}+1}\vee \Delta^{n_{p+2}} \vee... \vee  \Delta^{n_k} \to \\
 \Delta^{n_1} \vee ... \vee  \Delta^{n_p} \vee \Delta^{n_{p+1}} \vee
  \Delta^{n_{p+2}}\vee ...\vee \Delta^{n_k}
\end{multline*}
by applying the first co-degeneracy map to the $(p+1)$-th bead; and
\begin{equation}\label{rule2}
f \circ u \sim f,
\end{equation}
for any $f \in  Hom(T,X)_{x,y} $ and any morphism $u$ in $Nec$ of type $(iii)$. Denote  the $\sim$-equivalence class of $(f: T\to X)$ by
 $[f: T \to X]$.

\subsection{The necklical set $\mathbf{\Omega}(X;x,y)$ } Define a necklical set $\mathbf{\Omega} (X;x,y): Nec^{op} \to Set$ by declaring $\mathbf{\Omega} (X;x,y) (T)$ to be the subset of $\bigsqcup_{T'\in Nec} Hom(T',X)_{x,y}/\sim$ consisting of all $\sim$-equivalence classes represented by maps $T \to X \in (Nec \downarrow X)_{x,y}$.  This clearly defines a functor: given a map $u: T \to T'$ in $Nec$ and an element $[f: T' \to X] \in \mathbf{\Omega} (X;x,y) (T')$ we obtain a well defined element  $ [f \circ u: T \to T' \to X ] \in \mathbf{\Omega}(X;x,y)(T)$.
In particular, $\mathbf{\Omega} (X;x,y)=\{\mathbf{\Omega}_{n,k} (X;x,y)\}_{n\geq 0,k\geq 1}$ is bigraded with
$\mathbf{\Omega}_{n,k} (X;x,y):= \{T\to X \in (Nec^{n,k} \downarrow X)_{x,y}  \} / \sim $ (assuming that for a class $[f:T\rightarrow X]$ the representative map $f:T\rightarrow X$ is with minimal $k=b(T)$). Note that $\mathbf{\Omega} (X;x,y)$ is precisely the following colimit in the category of necklical sets
\begin{eqnarray*}
\mathbf{\Omega} (X;x,y)=  \underset{f: T \to X \in (Nec \downarrow X)_{x,y} } {\text{colim}} Y(T).
\end{eqnarray*}

We define the \textit{necklical face maps} of $\mathbf{\Omega}(X;x,y)$ to be the set maps
 \[
\begin{array}{ll}
d^1_i : \mathbf {\Omega} _{n,k}(X;x,y) \rightarrow \mathbf {\Omega}_{n-1,k}(X;x,y),\vspace{1mm}\\
 d^0_i : \mathbf {\Omega} _{n,k}(X;x,y) \rightarrow \mathbf {\Omega}_{n-1,k+1}(X;x,y) &    \text{for}\ \ i=1,...,n,
\end{array}
 \]
 given by $d^{\epsilon}_i[f: T \to X]= [f \circ d_{\epsilon}^{i,T}: T^{\epsilon}_i \to T \to X]$. It is straightforward to check that these maps are well defined, i.e. independent of representative, and satisfy the standard cubical relations:

\begin{equation}\label{c-relations} d_{j}^{\epsilon }d_{i}^{\epsilon' }=d_{i}^{\epsilon'}d_{j+1}^{\epsilon },\ \  i\leq j,\ \ \epsilon,\epsilon'\in\{0,1\}.
\end{equation}

Denoting by $n_j(T)$ the dimension of the $j$-th bead in any $T \in Nec^{n,k}$, let $n_{(r)}:=n_1(T)+\cdots +n_r(T)$ for
$1\leq r\leq k$ and $n_{(0)}$ :=0. Define the \textit{necklical degeneracy maps} of $\mathbf {\Omega} _{n,k}(X;x,y)$ to be the set maps
\[  \eta_j : \mathbf {\Omega} _{n,k}(X;x,y) \rightarrow \mathbf {\Omega}_{n+1,k}(X;x,y), \ \     j=1,...,n+k+1,
 \]
given by $\eta_j[f: T \to X]= [f \circ s^{j-1,T}: T_j \to T \to X]$, where
 \[T_j= \Delta^{n_1} \vee ... \vee \Delta^{n_{r-1} } \vee \Delta^{n_r +1 } \vee \Delta^{n_{r+1}}  \vee ... \vee \Delta^{n_l}\] when $n_{(r-1)} + 1 \leq j \leq n_{(r)}+1$.

The face and degeneracy maps satisfy the following identities:
\begin{equation}\label{c-relations2}
\begin{array}{rlll}
  d^{\epsilon}_{i}\eta_j & =
\left\{
  \begin{array}{llll}
   \eta_{j-1} d^{\epsilon}_{i}, &   i<j-r,   &n_{(r-1)}+1 <  j \leq n_{(r)}+1,           \vspace{1mm} \\
     \eta_{j} d^{\epsilon}_{i-1}, &  i>j-r+1, &n_{(r-1)}+1 \leq j < n_{(r)}+1, \\
  \end{array}
\right.
\vspace{1mm}

\\
  d^{\epsilon}_{i}\eta_j & =
\left\{
  \begin{array}{llll}
 Id, &  i=j-r                          \hspace{0.57in}   = n_{(r-1)}+1, &    \\
   \eta_{i}d^\epsilon_{i-1}, & i=j-r+1 \hspace{0.35in}= n_{(r-1)}+1, \vspace{1mm} \\
Id, & i=j-r,j-r+1                                      \neq  n_{(r-1)}+1, & \epsilon=1,     \\
d^{0}_{i+1}\eta_j, & i=j-r              \hspace{0.6in}               \neq  n_{(r-1)}+1 , & \epsilon=0, \vspace{1mm}\\
  \end{array}
\right.
\vspace{1mm}
\\
\\
\eta_i\eta_j  & =
     \eta_{j} \eta_{i-1},  \hspace{0.33in} i>j.
  \end{array}
   \end{equation}
In order to verify the above identities it is convenient to consider following combinatorial description of the necklical set $I^{n-1}:=Y(\Delta^n)$, which we call the \textit{standard $(n-1)$-cube}.
The faces of  the simplicial set $\Delta^n$ may be labeled by the subsets of the set $(0,1,...,n)$ as usual. The top face of the necklical set $I^{n-1}$ may be labeled by the expression
$[0,1,...,n]$ so that we may write
\[
\begin{array}{lllll}
  d^0_i ( [0,1,...,n] ) =[0,...,i][i,...,n],             \\
  d^1_i ( [0,1,...,n] )=  [0,1,...,\hat{i},...,n]  ,  & i=1,...,n-1,
\end{array}
\]
where $[0,...,i][i,...,n]$ denotes the face induced by $\Delta^i \vee \Delta^{n-i} \hookrightarrow \Delta^n$.
The description of  degeneracies $\eta_j:I^{n-1} \rightarrow I^{n}$ mimics the simplicial ones
\[\eta_j( [0,1,...,n] )= [0,1,...., j-2,j-1,j-1,j,...,n],\ \  j=1,...,n+1.\]

\subsection{The augmented necklical set $\mathbf{P}(X;x)$}
Define  an augmented necklical set $\mathbf{P} (X;x): Nec_0^{op} \to Set $ in a similar manner
by declaring $\mathbf{P} (X;x) (T)$ to be the subset of $\bigsqcup_{T'\in Nec_0} Hom(T',X)_{x}/\sim$ consisting of all $\sim$-equivalence classes represented by maps $T \to X \in (Nec_0 \downarrow X)_{x}$,  where $\sim$ is the equivalence relation given analogously to (\ref{rule1})--(\ref{rule2}). Similarly, define the \textit{augmented necklical face maps}
\[
\begin{array}{ll}
d^1_i : \mathbf {P} _{n,k}(X;x) \rightarrow \mathbf {P}_{n-1,k}(X;x), \vspace{1mm}\\
d^0_i : \mathbf {P} _{n,k}(X;x) \rightarrow \mathbf {P}_{n-1,k+1}(X;x)\ \ \text{for}\ \   i=1,...,n,
\end{array}
  \]
and \textit{augmented necklical degeneracy maps}
\[  \eta_j : \mathbf {P} _{n,k}(X;x) \rightarrow \mathbf {P}_{n+1,k}(X;x)\ \ \text{for}\ \  j=1,...,n+k+1  \]
satisfying
 the standard cubical relations given by
 (\ref{c-relations}) and
  \begin{equation}
\begin{array}{rlll}

 d^{\epsilon}_i\eta_j & =
\left\{
  \begin{array}{llll}
   \eta_{j-1} d^{\epsilon}_i,     &     \hspace{0.3in}  i< j,                \\
     \eta_{j} d^{\epsilon}_{i-1}, &      \hspace{0.3in}     i>j+1, \\
  \end{array}
\right.
\vspace{1mm}
\\
 d^1_{i}\eta_i   & =  d^1_{i+1}\eta_{i}  =Id,            \\
d^{0}_{i}\eta_i  & =  d^{0}_{i+1}\eta_i,       \hspace{1.5in}    1\leq i\leq n_0,  \\

\vspace{5mm}
\\
  d^{\epsilon}_{i}\eta_j & =
\left\{
  \begin{array}{llll}
   \eta_{j-1} d^{\epsilon}_{i}, &   i<j-r,       &n_{(r-1)}+1 <  j \leq n_{(r)}+1,           \vspace{1mm} \\
     \eta_{j} d^{\epsilon}_{i-1}, &  i >j-r+1,     &  n_{(r-1)}+1 \leq j < n_{(r)}+1, \\
  \end{array}
\right.
\vspace{1mm}

\\
  d^{\epsilon}_{i}\eta_j & =
\left\{
  \begin{array}{llll}
 Id, &  i=j-r                          \hspace{0.57in}   = n_{(r-1)}+1, &    \\
   \eta_{i}d^\epsilon_{i-1}, & i=j-r+1 \hspace{0.35in}= n_{(r-1)}+1, \vspace{1mm} \\
Id, & i=j-r,j-r+1                                      \neq  n_{(r-1)}+1, & \epsilon=1,     \\
d^{0}_{i+1}\eta_j, & i=j-r              \hspace{0.6in}               \neq  n_{(r-1)}+1 , & \epsilon=0, \vspace{1mm}\\
  \end{array}
\right.
\vspace{1mm}
\\
\\
\eta_i\eta_j  & =
     \eta_{j} \eta_{i-1},  \hspace{0.33in} i>j.
  \end{array}
   \end{equation}
 where $n_{(m)}=n_0+n_1+\cdots +n_m,\, 0\leq m\leq k.$
Note that these relations agree with (\ref{c-relations2}) for $n_0=0.$ The top cell of the augmented necklical set $I^n_{aug}:=Y_0(\Delta^n)$ (where $Y_0$ is the Yoneda embedding $Y_0: Nec_0 \to Set_{Nec_0}$) may be labeled by the symbol $0,1,..,n]$ so that we may write
\[
\begin{array}{lllll}
  d^0_i ( 0,1,..,n])=    0,...,i-1][i-1,...,n],            \vspace{1mm} \\

  d^1_i ( 0,1,..,n] ) =  0,1,...,\widehat{i-1},...,n]  ,  & i=1,...,n.
\end{array}
\]
The labeling of degeneracies remains the same.

Thus an $m$-dimensional cell $a$ of the augmented necklical set $I^n_{aug}$ is labelled by a sequence of blocks
\begin{multline}\label{faces}
a:=j_{1},...,j_{s_{1}}][j_{s_{1}},...,j_{s_{2}}]
[j_{s_{2}},...,j_{s_{3}}]...[j_{s_{q}},...,j_{s_{q+1}}]\ \ \text{with}\ \  j_{s_{q+1}}=n \ \ \text{and}\\
0\leq j_{1}<\ldots <j_{s_{q}}<n,
\end{multline}
where the dimension of the first block $j_{1},...,j_{s_{1}}]$ is $s_1-1,$ while the dimension of  each block $[j_{s_{k}},...,j_{s_{k+1}}]$  is $ s_{k+1}-s_k-1$, so $m=s_{q+1}-q-1.$
In particular, a vertex $v$ of $I^n_{aug}$ is labelled by
\[
v:=j_{1}][j_{{1}},j_{{2}}][j_{{2}},j_{{3}}]...[j_{q-1},j_q][j_{{q}},n]. \]
We have a map   $\psi: I^n_{aug} \to \Delta^n$ of graded sets which sends a face $a \in I^n_{aug} $ with labeling given as above to the $(s_1-1)$-simplex $(j_1,...,j_{s_1})\subset \Delta^n;$ in particular, the face
$a=d^0_1(0,1,...,n] )=0][0,1,...,n]$ is totally degenerate:
$\psi(a)=0\in \Delta^n.$ The two combinatorial descriptions of the faces of $I^n$ and $I^n_{aug}$, respectively, are compatible to one another
 via the combinatorial  analysis of the map $\psi.$

\subsection{The geometric realization of $\mathbf{\Omega} (X;x,y)$ and $\mathbf{P}(X;x)$}
The \textit{geometric realization} of the necklical set $\mathbf{\Omega} (X;x,y) $  is the topological space defined by
\begin{eqnarray*}
|\mathbf{\Omega} (X;x,y)|:= \bigsqcup_{n\geq 0}\mathbf{\Omega}_n (X;x,y)  \times | I^n |  / \sim
\end{eqnarray*}
where $| I^n |$ denotes the standard topological $n$-cube as a subspace of $\mathbb{R}^n$, $\mathbf{\Omega}_n (X;x,y) $ is considered as a topological space with the discrete topology, and $\sim$ is the equivalence relation defined as follows. The equivalence relation $\sim$ is generated by
$ (f, \partial^i_{\epsilon} (t))\sim (d^{\epsilon}_i(f), t) $  for $\partial^i_{\epsilon}: |I^{n-1}| \to |I^n|$ the usual cubical co-face maps, and $(f, \varsigma^j(t))\sim (\eta_j(f), t)$,
for any $f: \Delta^{n_1} \vee ... \vee \Delta^{n_k} \rightarrow X$, where  $\eta_j(f):= f \circ s^{j-1}$ and
\[s^j: \Delta^{n_1} \vee ... \vee \Delta^{n_p + 1} \vee ... \vee \Delta^{n_k} \to \Delta^{n_1} \vee ... \vee \Delta^{n_p} \vee ... \vee \Delta^{n_k}\]
is the morphism of necklaces obtained by applying the co-degeneracy morphism $s^j: \Delta^{n_p+1} \to \Delta^{n_p}$ to the $p$-th bead in the domain (for some $p \in \{1,...,k\}$), and $\varsigma^j: |I^{n+1}| \to |I^n|$ is the map induced by applying to the $p$-th factor in the decomposition $|I^{n+1}| =|I^{n_1-1}|\times ...\times | I^{n_p}| \times ... \times | I^{n_k-1}|$ the collapse map
$\varsigma^j: |I^{n_p}|\rightarrow |I^{n_p-1}|$ defined by
 \[ \varsigma^j (t_1,...,t_{n_p}) = (t_1,...,t_{j-1}, t_{j+1},...,t_{n_p} ) \] for $j=0,n_p+1$ (these are standard cubical degeneracies) and by
 \[ \varsigma^j (t_1,...,t_{n_p}) = (t_1,...,t_{j-1},\min(t_{j},t_{j+1}),t_{j+2},...,t_{n_p}) \] for $0<j<n_p+1$ (these are ``connections" in the sense of \cite{Brown- Higgins}).

We define the topological space $|\mathbf{P} (X;y)|$ in a completely analogous manner.

\begin{remark}
\normalfont The equality $f \circ s^{n_p}(T) = f\circ  s^{0}(S)$ in the definition of $\mathbf{\Omega} (X;x,y)$  is suggested by the fact  that  $|Y(T)|$ and $|Y(S)|$   define homeomorphic  cubes  and
this homeomorphism is compatible with the equality   $ |Y(s^{n_p}(T))|= |Y( s^{0}(S))| .$
\end{remark}

\subsection{Inverting $1$-simplices formally}

Given a simplicial set $X,$ form a set $X_1^{op}:=\{ x^{op} \mid x \in X_1\ \ \text{is non-degenerate} \}.$ Let $Z(X)$ be the minimal simplicial set containing the set
$X\cup X^{op}_1$ such that $\partial_0( x^{op} )= \partial_1( x)$ and  $\partial_1( x^{op} ) = \partial_0 (x).$
  Let $Set_{\Delta}^0$ be the category of pointed simplicial sets. Define
\begin{eqnarray*}
\widehat{\mathbf{\Omega}}: Set^0_{\Delta} \to Set_{Nec}
\end{eqnarray*}
as $\widehat{\mathbf{\Omega}}(X;x,y):= \mathbf{\Omega}(Z(X);x,y) / \sim$
where the equivalence relation $\sim$ is generated by
\[ f\sim  f'\circ g:T\rightarrow Z (X) \]
for all $T=\Delta^{n_{1}}\vee ...\vee\Delta^{n_p}\vee\Delta^{n_{p+1}}\vee ...\vee \Delta^{n_{k}} $ with $n_p=n_{p+1}=1$ and
  morphisms $f$ satisfying $f(\Delta^{n_p})=(f(\Delta^{n_{p+1}}))^{op},$ so $f$ induces a map
$f':  \Delta^{n_{1}}\vee ...\vee \Delta^{n_{p-1}}\vee\Delta^{n_{p+2}}\vee ...\vee \Delta^{n_{k}}\rightarrow X,$
and
\begin{multline*}
g: \Delta^{n_{1}}\vee ...\vee \Delta^{n_{p-1}}\vee\Delta^{n_p}\vee\Delta^{n_{p+1}}\vee\Delta^{n_{p+2}}\vee ...\vee \Delta^{n_{k}}
\rightarrow
\\
 \Delta^{n_{1}}\vee ...\vee \Delta^{n_{p-1}}\vee\Delta^{n_{p+2}}\vee ...\vee \Delta^{n_{k}}
\end{multline*}
is the collapse map.

The above equivalence relation also induces the functor
\begin{eqnarray*}
\widehat{\mathbf{P}}: Set^0_{\Delta} \to Set_{Nec_0}
\end{eqnarray*}
by replacing ${\mathbf{\Omega}}$  by $\widehat{\mathbf{\Omega}}$ in the definition of   ${\mathbf{P}}.$
 For simplicity we denote $\widehat{\mathbf{\Omega}}X:=\widehat{\mathbf{\Omega}}(X;x,x)$ and $\widehat{\mathbf{P}}X:=
\widehat{\mathbf{P}}(X;x)$ when the fixed point $x$ is understood.

\subsection{An explicit construction of $\widehat{\mathbf{\Omega}}X$  and $\widehat{\mathbf{P}}X$}

We now describe $\widehat{\mathbf{\Omega}}X$ and $\widehat{\mathbf{P}}X$ more concretely. Let $(X,x_0)$ be a pointed simplicial set. For a simplex $\sigma \in X$ denote by $\min \sigma$ and $\max \sigma$ the first and last vertices of $\sigma$, respectively. First define  $\widehat{\mathbf{\Omega}}'X$  to be the following graded set. For any $\sigma_i\in Z(X)_{>0}$, let $\dim (\bar \sigma)=\dim(\sigma)-1 $ and define
 \begin{multline}
\widehat{\mathbf{\Omega}}'_n X=\{ \bar \sigma_1\cdots \bar \sigma_k  \mid \max \sigma_i=\min \sigma_{i+1}\
\text{for all}\ i,\ \max\sigma_k= x_0 ,\\  \, | \bar \sigma_1|+\cdots +| \bar \sigma_k|=n,\, k\geq 1  \}
  \end{multline}
with relations
\[\bar \sigma_1\cdots  \bar \sigma_i\cdot \bar \sigma_{i+1}   \cdots \bar \sigma_k=
\bar \sigma_1\cdots \bar \sigma_{i-1}\cdot \bar \sigma_{i+2}\cdots    \bar \sigma_k
\ \ \text{and}\ \  \bar \sigma_i\cdot \bar \sigma_{i+1}= \overline{s_0( x)}
 \]
where $\sigma_{i},\sigma_{i+1}\in Z(X)_1$ such that $\sigma_{i+1}=\sigma^{op}_i$ and $x=\min \sigma_i$; and
\[\bar \sigma_1\cdots  \overline{s_{n_i}( \sigma_i)}\cdot \bar \sigma_{i+1}\cdots \bar \sigma_k=
\bar \sigma_1\cdots  \bar \sigma_{i}\cdot
\overline{s_0( \sigma_{i+1})} \cdots \bar \sigma_k\ \ \text{for}   \ \ n_i=\dim\sigma_i,\,i=1,...,k-1.\]
Then
\[\widehat{\mathbf{\Omega}}_n X \cong \{ \bar \sigma_1\cdots \bar \sigma_k \in
\widehat{\mathbf{\Omega}}'_n X \mid  \min \sigma_{1}=x_0\}.\]
and the monoidal structure $ \widehat{\mathbf{\Omega}}X \times  \widehat{\mathbf{\Omega}}X \to  \widehat{\mathbf{\Omega}}X$ is induced by concatenation of words with unit $e=s_0( x_0).$
In particular,     $ \widehat{\mathbf{\Omega}}_0 X$ is a group.

The set $ \widehat{\mathbf{\Omega}}_nX$  has the second grading
$ \widehat{\mathbf{\Omega}}_n X=
\{\widehat{\mathbf{\Omega}}_{n,k} X\}_{k\geq 1},$
where $\widehat{\mathbf{\Omega}}_{n,k}$
consists  of words of length $k$ having the same dimension $n.$
The face operators
 \[
\begin{array}{ll}
d^1_i: \widehat{\mathbf{\Omega}}_{n,k} X \rightarrow \widehat{\mathbf{\Omega}}_{n-1,k} X,\\
d^0_i: \widehat{\mathbf{\Omega}}_{n,k} X \rightarrow \widehat{\mathbf{\Omega}}_{n-1,k+1} X, \ \ \ 1\leq i\leq n,
\end{array}
 \]
act on a monomial $a\bar \sigma b$ as
 \begin{multline*}
d^\epsilon_{j}(a\bar \sigma b)=a d^\epsilon_i(\bar \sigma) b\ \
\text{for}\ \  i=j-\dim a,\, \dim a<j \leq \dim a +m,\,m+1=\dim \sigma\ \ \text{with} \\
                  \begin{array}{lllll} d^0_i(\bar \sigma) =\, \overline{\partial _{i+1}\cdots
                  \partial _{m+1}(\sigma)}\cdot \overline {\partial _0\cdots \partial _{i-1}(\sigma)}, \vspace{1mm}\\
                 d^1_i(\bar \sigma) = \,
                 \overline {\partial _i(\sigma)},& 1\leq i\leq m,
\end{array}
               \end{multline*}

while
\[\eta_j: \widehat{\mathbf{\Omega}}_{n,k} X \rightarrow \widehat{\mathbf{\Omega}}_{n+1,k} X\ \  \text{for} \ \
j=1,..,n+k+1
 \]
is given by
  \begin{multline*}
\eta_j(\bar \sigma_1 \cdots  \bar \sigma_k)=
 \bar \sigma_1\cdots \bar \sigma_{i-1}\cdot  \overline{s_{j-m_{i-1}}( \sigma_i)}\cdot \bar \sigma_{i+1}\cdots \bar \sigma_k \ \
\text{for}\ \
       m_{i-1}\leq j\leq m_{i},\\
       m_{i}=\dim\sigma_1+\cdots +\dim\sigma_ {i}\ \text{with}\ m_0=1,
 1\leq i\leq k.
\end{multline*}

Similarly, the graded set  $\widehat{\mathbf{P}} X=\{\widehat{\mathbf{P}}_n X\}_{n\geq 0}$ may be identified with $\mathbf{P}' X/\sim, $
where
${\mathbf{P}}' X$ is   a subset of the (set-theoretical) cartesian product
   $X\times \widehat{\mathbf{\Omega}}'X$  of two graded sets $X$ and $\widehat{\mathbf{\Omega}}'X:$
\begin{equation}
{\mathbf{P}}'_n X=
\{ \left(x\,,\, \bar \sigma_1\cdots \bar \sigma_k \right) \in\bigcup_{p+q=n} X_{p}\times \widehat{\mathbf{\Omega}}'_{q} X \mid \max x=\min \sigma_1\}
\end{equation}
and $\sim$ is defined by
setting the relation
\[ (s_{p}(x),y)\sim (x,\eta_1(y)). \]
The face operators
 \[d^0_i,d^1_i: \widehat{\mathbf{P}}_n X \rightarrow
 \widehat{\mathbf{P}}_{n-1} X, \ \ \ 1\leq i\leq n, \]
  are  then given
  for
 $ (x\,,  y)\in X_p\times \widehat{\mathbf{\Omega}}'_q X
\rightarrow   \widehat{\mathbf{P}}_n X    $ by
\[
 \begin{array}{llllll}
 d^0_i(x,y)= \left\{
   \begin{array}{llll}
  \left (\min x\, ,\,    \bar x  \cdot y \right), &  & i=1 ,
   \newline \vspace{1mm}\\
   \left( {\partial _{i}\cdots \partial _{p+1}(x)}\, ,\, \overline {\partial _0\cdots \partial _{i-2}(x)}    \cdot y \right), & &  2\leq i \leq p,           \newline \vspace{1mm}  \\
  (x\, , d^0_{i-p}(y))   , & &  p<i\leq n,
   \end{array}
 \right.
 \\
 \\
 d^1_i(x,y) = \left\{
   \begin{array}{llllllllll}
   ( \partial _{i-1}(x)\, ,  y), &  \hspace{1.50in}  1\leq i \leq p,\newline \vspace{1mm}\\
   (x ,d^1_{i-p}(y))   , &        \hspace{1.50in}   p<i\leq n,
   \end{array}
 \right.
 \end{array}
\]
and
the degeneracy maps
\[  \eta_j :  \widehat{\mathbf{P}} _{n,k}(X) \rightarrow
 \widehat{\mathbf{P}}_{n+1,k}(X)\ \ \text{for}\ \  j=1,...,n+k+1  \]
by
\begin{equation*}
\eta_j(x,y)=\left\{ \begin{array}{llll} (s_{j-1}(x), y),  &1\leq j\leq p+1, \\
                       (x, \eta_{j-p}(y)), & p+1\leq j\leq n+1.
\end{array}
\right.
\end{equation*}

Furthermore, concatenation of words gives a map of augmented necklical sets
 \begin{equation}\label{action}
 \widehat{\mathbf{P}}X\times  \widehat{\mathbf{\Omega}}X\rightarrow  \widehat{\mathbf{P}}X     .
 \end{equation}
We have  the  short sequence
\[
 \widehat{\mathbf{\Omega}} X   \overset{i }{\longrightarrow}
  \widehat{\mathbf{P}}X   \overset{pr}{\longrightarrow} X
\]
of maps of sets where $i$ is defined by  $i(y)=(x_0, y)$, for any $y \in \widehat{\mathbf{\Omega}} X $
while $pr(x,y)=x$ for
$(x,y)\in   \widehat{\mathbf{P}} X.$
The set map $i: \widehat{\mathbf{\Omega}} X \to \widehat{\mathbf{P}}X$ induces a continuous map $ \iota: |\widehat{\mathbf{\Omega}} X|  \to |\widehat{\mathbf{P}}X |$. The projection
$pr : \widehat{\mathbf{P}}X  {\longrightarrow} X$ together with the continuous map $|\psi|:|I^n|\rightarrow |\Delta^n|$ (induced by the maps of graded sets $\psi: I^n_{aug} \rightarrow \Delta^n$)
induce a continuous map $\xi: |\widehat{\mathbf{P}}X|  \to |X|$. More precisely, $\xi: |\widehat{\mathbf{P}}X | \to |X|$ is defined as the composition
\[ |\widehat{\mathbf{P}}X | \overset{Id\times pr_X} \longrightarrow |\widehat{\mathbf{P}}X |   \overset{pr \times |\psi| } \longrightarrow |X| \]
where $Id\times pr_X$ projects a cell $[(\sigma,y)\times |I^n|\times |I^q|]$ onto $[ (\sigma,y)\times |I^n|]$ and $pr\times |\psi| $ maps any cell $[(\sigma,y)\times |I^N|]$ onto $[\sigma \times |\Delta^{N}|]$. Thus $\xi$ is a continuous cellular map and
we have
\begin{proposition}\label{quasi} Let $(X,x_0)$ be a pointed connected simplicial set. Then

(i) the geometric realization $|\widehat{\mathbf{\Omega}} X|$ is a topological monoid,

(ii) the geometric realization $|\widehat{\mathbf{P}} X|$ is contractible, and

(iii) the geometric realization of $\xi$
\[
   |\widehat{\mathbf{\Omega}} X|   \overset{\iota}{\longrightarrow}
    |\widehat{\mathbf{P}}X|   \overset{\xi}{\longrightarrow} |X|
\]

is  a quasi-fibration.
\end{proposition}
\begin{proof}
(i) Immediately follows from the definition of the geometrical realization of the monoid $\widehat{\mathbf{\Omega}} X.$

(ii) A contraction of $|\widehat{\mathbf{P}} X|$ into the vertex labeled by $(x_0,e)$ may be obtained as follows.
First note that the $1$-dimensional subcomplex
$|\widehat{\mathbf{P}} X^1|$ of $|\widehat{\mathbf{P}} X|$ is contractible, and below we deform
$|\widehat{\mathbf{P}} X|$ into $|\widehat{\mathbf{P}} X^1|.$

Given $(x,y)\in |\widehat{\mathbf{P}} X| $ with $x\in X_n,$ let $|\hat x^{n,1}|$ denote
  the  union of edges of
    $| (x, y)|\in  |\widehat{\mathbf{P}} X| $  not  lying in $|d^0_1(x,y)|=|(\min x, \bar x\cdot \bar y)|.$

Then a cube $|(x, \bar \sigma_1\cdots \bar \sigma_k )|$  in $|\widehat{\mathbf{P}}  X|$ can be deformed
into the subcomplex $|\hat x^{n,1}|\cup |\hat \sigma_1^{n_1,1}|\cup ...\cup |\hat \sigma_k^{n_k,1}|$  of   $|\widehat{\mathbf{P}} X^1|,$
where  $|\hat \sigma_i^{n_i,1}|$ is defined by the cube $|(\sigma_i, \bar \sigma_{i+1}\cdots \bar \sigma_k )|$ for all $i.$ These deformations may be defined so that they are compatible at the intersection of any two cubes in $|\widehat{\mathbf{P}}  X|$ and therefore induce a global contraction of  $|\widehat{\mathbf{P}}  X|$ into the vertex $|(x_0,e)|$.

(iii)
 Recall $|X|$ is a space defined as a colimit of standard topological simplices with identifications given by the face and dengeneracy maps of $X$. Take the barycentric subdivision of each standard simplex in the colimit to obtain a finer subdivision of $|X|$ into simplices. For each simplex $\sigma \subset |X|$ in this subdivision let $U_{\sigma}$ be an open neighborhood containing $\sigma$ as a deformation retract; in particular, each $U_{\sigma}$ is contractible.
  Let $\mathcal{U}$ be the smallest collection of open sets containing $\{U_{\sigma} \}$ which is closed under finite intersections. Then $\mathcal{U}$ is an open covering of $|X|$ with the property that for any $U \in \mathcal{U}$ and any $x \in U$, $\zeta^{-1}(x)\hookrightarrow \zeta^{-1}(U)$ is a homotopy equivalence.
  It follows that $\xi$ satisfies the criterion in \cite{Dold- Thom}  to be a quasi-fibration.
\end{proof}
From the long exact sequence of the homotopy groups of a quasi-fibration and the contractibility of $|\widehat{\mathbf{P}}  X|$ we have
\begin{corollary}\label{coro}
For a  simplicial set $X$ there is a natural isomorphism
\[  \pi_{\ast}(|X|)\approx  \pi_{\ast-1} (|\widehat{\mathbf{\Omega}} X|) . \]
\end{corollary}
\begin{example}
\normalfont Let $Y$ be the boundary of $\Delta^2.$ Then $\xi$ can be thought of as an extension of a simplicial approximation of the exponential map
$exp=e^{2 \pi i t}:{\mathbb R}\rightarrow S^1.$ See Figure 1 below where
$\alpha_{012}=[02][21][10]$ and $\alpha_{210}=[01][12][20]$ are identified with  representatives of the fundamental group of the circle.

\newpage

\unitlength 1mm 
\linethickness{0.4pt}
\ifx\plotpoint\undefined\newsavebox{\plotpoint}\fi 
\begin{picture}(114.27,47.267)(0,0)
\put(65.336,41.32){\makebox(0,0)[cc]{$_{_{02][21][10]}}$}}
\put(18.623,41.32){\makebox(0,0)[cc]{$_{_{12][20]}}$}}
\put(7.836,41.32){\makebox(0,0)[cc]{$_{_{01][12][20]}}$}}
\put(42.036,41.32){\makebox(0,0)[cc]{$_{_{01][10]}}$}}
\put(83.157,44.439){\line(1,0){25.456}}
\put(108.543,44.369){\line(1,0){.9428}}
\put(110.428,44.369){\line(1,0){.9428}}
\put(112.314,44.369){\line(1,0){.9428}}
\put(13.887,46.871){\makebox(0,0)[cc]{$1$}}
\put(25.312,46.871){\makebox(0,0)[cc]{$2$}}

\put(36.312,46.871){\makebox(0,0)[cc]{$0$}}

\put(48.787,46.871){\makebox(0,0)[cc]{$1$}}
\put(60.587,46.871){\makebox(0,0)[cc]{$2$}}
\put(83.187,46.871){\makebox(0,0)[cc]{$1$}}
\put(94.787,46.871){\makebox(0,0)[cc]{$2$}}
\put(36.135,37.898){\vector(0,-1){8.839}}
\put(30.301,24.993){\line(1,0){13.789}}
\multiput(44.09,24.993)(-.044195652,-.033625){184}{\line(-1,0){.044195652}}
\multiput(35.958,18.806)(-.033672619,.03577381){168}{\line(0,1){.03577381}}
\put(27.5,25.0){\makebox(0,0)[cc]{$0$}}
\put(46.9,25.0){\makebox(0,0)[cc]{$2$}}

\put(32.953,33.832){\makebox(0,0)[cc]{$\xi$}}
\put(41.67,33.371){\makebox(0,0)[cc]{$exp$}}
\put(71.137,47.267){\makebox(0,0)[cc]{$\alpha_{012}$}}

\put(105.078,48.0){\makebox(0,0)[cc]{$\alpha_{012}^2$}}
\put(3.077,47.267){\makebox(0,0)[cc]{$\alpha_{210}$}}
\put(93.712,44.496){\vector(1,0){.375}}
\put(96.212,44.496){\vector(-1,0){.5}}

\put(30.272,24.807){\circle*{1.282}}
\put(44.358,24.807){\circle*{1.282}}
\put(35.949,18.71){\circle*{1.282}}

\put(83.039,44.358){\line(-1,0){83.459}}
\put(-.281,44.287){\line(-1,0){.8409}}
\put(-1.962,44.287){\line(-1,0){.8409}}
\put(-3.644,44.287){\line(-1,0){.8409}}
\put(36.369,44.358){\circle*{1.682}}

\put(48.772,44.358){\circle*{1.1}}
\put(60.545,44.358){\circle*{1.1}}

\put(71.056,44.358){\circle*{1.682}}
\put(105.113,44.568){\circle*{1.682}}

\put(25.227,44.358){\circle*{1.1}}
\put(13.875,44.358){\circle*{1.1}}

\put(3.153,44.358){\circle*{1.682}}

\put(83.249,44.5){\circle*{1.1}}
\put(94.812,44.5){\circle*{1.1}}

\put(58.653,44.5){\vector(1,0){1.051}}
\put(46.67,44.5){\vector(1,0){1.682}}
\put(62.437,44.5){\vector(-1,0){1.051}}
\put(81.357,44.5){\vector(1,0){.841}}
\put(27.119,44.5){\vector(-1,0){1.051}}
\put(22.915,44.5){\vector(1,0){1.051}}
\put(11.562,44.5){\vector(1,0){1.261}}

\put(35.949,15.6){\makebox(0,0)[cc]{$1$}}
\end{picture}

\begin{center}
 Figure 1.
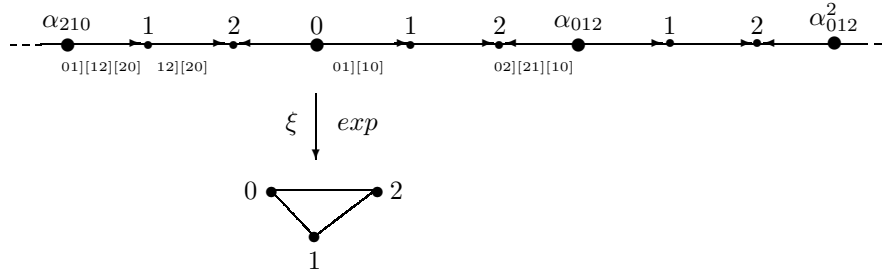
The map $\xi$  on $ \partial \Delta^2 \approx  S^1$  as a simplicial approximation of the exponential  map.
\end{center}

\end{example}

\vspace{0.2in}

More generally, let $Y$ be one dimensional, i.e.,
it is homotopically equivalent to a wedge of circles.
 See
Figure 2 below
where
$\alpha_{ijk}$ and $\alpha_{kji}$ denote opposite representatives of the fundamental group given by
the closed edge path  $((ij),(jk),(ki))$ in  $Y:=|\Delta^3|^1,$ the $1$-skeleton of $|\Delta^3|,$ homotopically equivalent to the wedge of three circles.


\newpage
\unitlength .9mm 
\linethickness{0.4pt}
\ifx\plotpoint\undefined\newsavebox{\plotpoint}\fi 
\begin{picture}(584.231,153.7)(0,0)
\put(40.712,112.18){\line(1,0){53.366}}
\multiput(94.078,112.18)(.0470268097,.037461126){373}{\line(1,0){.0470268097}}
\put(111.619,126.153){\line(0,1){22.149}}
\multiput(111.619,148.302)(-.0443422983,-.0374352078){409}{\line(-1,0){.0443422983}}
\multiput(93.483,132.991)(-6.24325,.03725){4}{\line(-1,0){6.24325}}
\put(68.51,133.14){\line(-1,0){.149}}
\put(111.619,148.302){\line(-1,0){28.541}}
\put(111.47,126.004){\line(-1,0){27.798}}
\multiput(40.712,112.031)(-.0439755501,-.0374352078){409}{\line(-1,0){.0439755501}}
\put(22.874,96.572){\line(1,0){27.5}}
\multiput(50.077,96.572)(.3345,.037){4}{\line(1,0){.3345}}
\put(51.415,96.72){\line(1,0){.743}}
\put(22.726,74.274){\line(1,0){28.987}}
\multiput(40.861,89.882)(6.7635,-.037){4}{\line(1,0){6.7635}}
\put(68.807,133.14){\circle*{1.516}}
\put(68.361,112.109){\circle*{1.601}}
\put(67.321,89.734){\circle*{1.516}}
\put(51.713,96.572){\circle*{1.601}}
\put(51.118,74.125){\circle*{1.516}}
\put(84.118,125.856){\circle*{1.516}}
\put(83.524,148.302){\circle*{1.601}}
\put(94.078,111.883){\circle*{.94}}
\put(111.322,125.856){\circle*{.94}}
\put(111.47,148.302){\circle*{1.189}}
\put(40.861,111.883){\circle*{.892}}
\put(40.712,111.883){\circle*{.892}}
\put(40.861,111.883){\circle*{1.189}}
\put(22.874,96.572){\circle*{.892}}
\put(22.726,74.125){\circle*{.892}}
\put(23.023,96.572){\circle*{1.226}}
\put(67.321,89.734){\circle*{1.784}}
\put(51.267,74.125){\circle*{1.601}}
\put(68.659,133.14){\circle*{1.487}}
\put(68.51,132.991){\circle*{1.784}}
\put(83.821,125.856){\circle*{1.487}}
\put(93.632,112.031){\vector(1,0){.078}}\multiput(91.7,111.883)(.483,.037){4}{\line(1,0){.483}}
\put(110.876,125.41){\vector(2,1){.078}}\multiput(109.835,124.815)(.0650625,.0371875){16}{\line(1,0){.0650625}}
\put(111.47,147.856){\vector(0,1){.078}}\multiput(111.619,146.518)(-.03725,.3345){4}{\line(0,1){.3345}}
\put(108.943,148.302){\vector(1,0){1.635}}
\put(92.889,133.14){\vector(1,0){.078}}\multiput(90.956,132.991)(.48325,.03725){4}{\line(1,0){.48325}}
\put(110.727,147.856){\vector(3,4){.078}}\multiput(109.984,146.816)(.03715,.052){20}{\line(0,1){.052}}
\put(39.969,111.734){\vector(3,4){.078}}\multiput(39.226,110.693)(.03715,.05205){20}{\line(0,1){.05205}}
\put(40.712,89.734){\circle*{.892}}
\put(93.483,132.842){\circle*{.892}}
\put(112.362,149.491){$3$}
\put(82.768,150.109){$0$}
\put(67.868,134.923){$0$}

\put(82.868,127.109){$0$}
\put(67.568,113.909){$0$}
\put(113.254,125.856){$2$}

\put(40.368,113.909){$2$}
\put(41.818,87.609){\makebox(0,0)[cc]{$3$}}
\put(50.868,98.109){$0$}
\put(66.668,91.609){$0$}
\put(50.368,76.058){$0$}

\put(92.0,134.626){$1$}
\put(93.0,113.909){$1$}

\put(20.199,73.0){$1$}
\put(20.199,95.5){$1$}

\put(110.43,126.153){\vector(4,1){.078}}\multiput(109.24,125.856)(.14875,.037125){8}{\line(1,0){.14875}}
\put(79.956,110.099){\makebox(0,0)[cc]{$_{_{01][10]}}$}}
\put(109.443,117.68){\makebox(0,0)[cc]{$_{_{12][21][10]}}$}}
\put(96.456,124.369){\makebox(0,0)[cc]{$_{_{02][21][10]}}$}}

\put(121.521,136.109){\makebox(0,0)[cc]{$_{_{23][32][21][10]}}$}}

\put(98.521,146.609){\makebox(0,0)[cc]{$_{_{03][32][21][10]}}$}}
\put(81.368,131.609){\makebox(0,0)[cc]{$_{_{01][13][32][21][10]}}$}}
\put(92.521,140.572){\makebox(0,0)[cc]{$_{_{13][32][21][10]}}$}}
\put(52.902,110.309){\makebox(0,0)[cc]{$_{_{02][20]}}$}}
\put(27.064,105.045){\makebox(0,0)[cc]{$_{_{12][20]}}$}}
\put(35.807,72.193){\makebox(0,0)[cc]{$_{_{01][13][32][20]}}$}}
\put(29.168,86.109){\makebox(0,0)[cc]{$_{_{13][32][20]}}$}}
\put(48.168,103.598){\makebox(0,0)[cc]{$_{_{23][32][20]}}$}}
\put(52.902,88.098){\makebox(0,0)[cc]{$_{_{03][32][20]}}$}}
\put(33.118,94.609){\makebox(0,0)[cc]{$_{_{01][12][20]}}$}}
\put(75.715,147.926){\line(0,1){.0745}}
\put(67.242,133.061){\line(-1,0){.892}}
\put(65.458,133.061){\line(-1,0){.892}}
\put(63.674,133.061){\line(-1,0){.892}}
\put(61.89,133.061){\line(-1,0){.892}}
\put(82.404,125.777){\line(-1,0){.9344}}
\put(80.536,125.777){\line(-1,0){.9344}}
\put(78.667,125.777){\line(-1,0){.9344}}
\put(76.798,125.777){\line(-1,0){.9344}}
\put(52.377,96.345){\line(1,0){.9556}}
\put(54.289,96.345){\line(1,0){.9556}}
\put(56.2,96.345){\line(1,0){.9556}}
\put(58.111,96.345){\line(1,0){.9556}}
\put(68.282,89.507){\line(1,0){.8707}}
\put(70.024,89.507){\line(1,0){.8707}}
\put(71.765,89.507){\line(1,0){.8707}}
\put(73.507,89.507){\line(1,0){.8707}}
\put(51.931,74.196){\line(1,0){.998}}
\put(53.927,74.153){\line(1,0){.998}}
\put(55.923,74.11){\line(1,0){.998}}
\put(57.919,74.068){\line(1,0){.998}}
\put(67.368,58.109){\makebox(0,0)[cc]{$\xi$}}
\put(70.957,64.487){\vector(0,-1){13.081}}
\multiput(71.311,42.92)(.0368125,-.72183333){48}{\line(0,-1){.72183333}}
\put(51.865,19.232){\line(1,0){19.445}}
\put(74.139,19.232){\line(1,0){17.324}}
\multiput(-112.944,-39.487)(-.0723993174,.0374061433){293}{\line(-1,0){.0723993174}}
\multiput(-284.772,-106.663)(-.0723993174,.0374061433){293}{\line(-1,0){.0723993174}}
\multiput(73.078,8.272)(-.0723993174,.0374061433){293}{\line(-1,0){.0723993174}}
\put(49.368,19.009){\makebox(0,0)[cc]{$0$}}
\put(93.352,19.009){\makebox(0,0)[cc]{$2$}}
\put(71.311,46.163){\makebox(0,0)[cc]{$3$}}

\multiput(-112.748,-39.487)(.062537931,.0374172414){290}{\line(1,0){.062537931}}
\multiput(-284.576,-106.663)(.062537931,.0374172414){290}{\line(1,0){.062537931}}
\multiput(73.274,8.272)(.062537931,.0374172414){290}{\line(1,0){.062537931}}
\multiput(91.113,19.569)(.037125,-.074375){8}{\line(0,-1){.074375}}
\multiput(71.342,42.907)(-.0374515504,-.0458042636){516}{\line(0,-1){.0458042636}}
\put(81.118,145.859){\makebox(0,0)[cc]{$\alpha_{_{023}}\alpha_{_{012}}$}}
\put(67.118,130.609){\makebox(0,0)[cc]{$\alpha_{_{123}}$}}
\put(82.118,123.359){\makebox(0,0)[cc]{$\alpha_{_{012}}$}}
\put(53.118,93.859){\makebox(0,0)[cc]{$\alpha_{_{210}}$}}
\put(69.868,86.609){\makebox(0,0)[cc]{$\alpha_{_{023}}$}}
\put(54.118,71.109){\makebox(0,0)[cc]{$\alpha_{_{310}}\alpha_{_{023}}$}}
\put(40.868,111.859){\line(0,-1){14.25}}
\put(40.868,95.109){\line(0,-1){5}}
\put(41.259,89.943){\vector(-4,1){.078}}\multiput(42.496,89.766)(-.2474,.0354){5}{\line(-1,0){.2474}}
\put(23.935,96.484){\vector(-4,1){.078}}\multiput(25.349,96.307)(-.2828,.0354){5}{\line(-1,0){.2828}}
\put(24.642,74.386){\vector(-1,0){1.237}}
\put(41.789,112.04){\vector(-1,0){.078}}\multiput(43.38,112.217)(-.3182,-.0354){5}{\line(-1,0){.3182}}
\put(40.917,91.185){\vector(0,-1){.53}}
\put(40.669,89.528){\vector(4,3){.078}}\multiput(22.379,74.392)(.0452722772,.0374653465){404}{\line(1,0){.0452722772}}
\put(82.635,148.313){\line(-1,0){.841}}
\put(80.953,148.313){\line(-1,0){.841}}
\put(79.271,148.313){\line(-1,0){.841}}
\multiput(67.867,112.429)(-.0374316547,.0393219424){556}{\line(0,1){.0393219424}}
\multiput(46.766,134.634)(-.03505,.03855){12}{\line(0,1){.03855}}
\multiput(45.925,135.559)(-.03505,.03855){12}{\line(0,1){.03855}}
\multiput(45.084,136.484)(-.03505,.03855){12}{\line(0,1){.03855}}
\put(51.0,123.934){\makebox(0,0)[cc]{$_{_{03][30]}}$}}
\multiput(71.382,42.843)(.037427426,-.0447792418){529}{\line(0,-1){.0447792418}}
\put(52.021,19.259){\circle*{1.5}}
\put(73.271,8.259){\circle*{1.5}}
\put(91.021,19.259){\circle*{1.5}}
\put(71.271,42.509){\circle*{1.5}}
\put(48.521,132.509){\circle*{1.5}}
\put(50.021,135.509){\makebox(0,0)[cc]{$3$}}
\put(73.447,4.772){\makebox(0,0)[cc]{$1$}}
\end{picture}

\vspace{0.2in}

\begin{center} Figure 2.
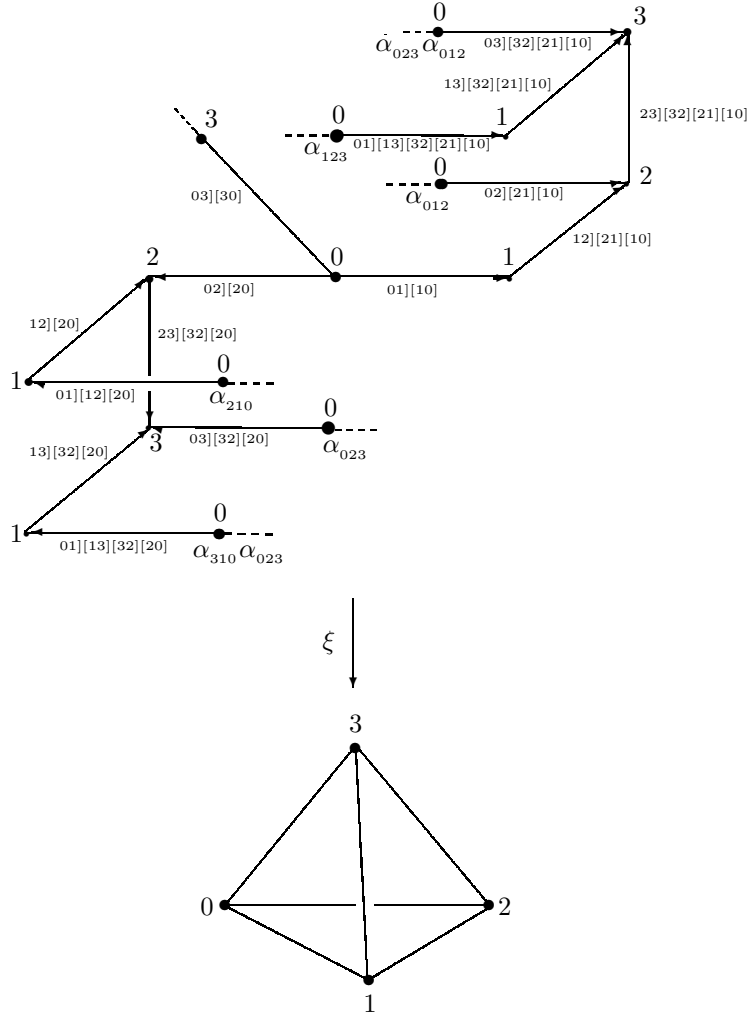
The  map $\xi$   for  $X=(\Delta^3)^1.$
\end{center}

\vspace{0.2in}

From now on we denote by $I^n$ the standard $n$-cube as a topological space since there will be no more risk of confusion with the standard $n$-cube as a necklical set. Similarly, we denote by $\Delta^n$ the standard $n$-simplex as a topological space, and simply denote by $\psi: I^n \to \Delta^n$ the continuous map denoted earlier by $|\psi|$.


\begin{definition}
Let $E$ be an equivariant space $E\times G\rightarrow E$ such that $G$ is a topological monoid.
A map    $g :(X,x_0)\times (Y,y_0)\rightarrow E$ is \textbf{equivariant} if there is a map
$\omega: (Y,y_0) \rightarrow (G,e)$ such that
\[g(x,y)=g(x,y_0)\cdot \omega (y)\ \ \text{for}  \ \  (x,y)\in X\times Y.\]
For a pair  $(X,Y)$ of subcomplexes of standard cubes $(I^n,I^q),$ the pair $(x_0,y_0):=(\mathbf{1},\mathbf{0})$ is assumed to consist
of
 the final and initial  vertices of $I^n $ and $I^q $ respectively.
A map $f:I^p\rightarrow E$ is
 \textbf{$d^0$-equivariant} if for each cell $I^k\subseteq I^p$
the restriction of $f$ to $d^0_i(I^k)=I^{i-1}\times I^{k-i}$
is equivariant
for  all   $i.$
\end{definition}

Let $ \pi: E \rightarrow Y$ be a $G$-fibration with $Y=|X|$ for a simplicial set $X.$
Let $\hat I^{n,1}$ be the  union of edges of $I^{n}$  not  lying in $d^0_1(I^n)$ and $\hat{i}:\hat I^{n,1}\hookrightarrow I^n$ the inclusion map.
For $q\geq 0,$ let $pr_n:I^n\times I^q\rightarrow I^n $ be the projection.
\begin{proposition}\label{lift}
Given $n\geq 2$ and a simplex $\sigma^n\in X_n,$ let  $\sigma^n: \Delta^n\rightarrow Y$ be  the corresponding map, and
let  $\rho_{n,1}: \hat I^{n,1}\times I^q\rightarrow E$  be an equvariant map
 such that
\[ \pi \circ \rho_{n,1} =  f_{\sigma^n}\circ  (\, \hat i\times 1)\ \  \text{for}\  \  f_{\sigma^n}:=\sigma^n\circ \psi \circ pr_n.    \]
 There is an equivariant and $d^0$-equivariant map $ g_{\sigma^n} :I^n\times I^q\rightarrow E$
such that
\[     \pi\circ g_{\sigma^n} = f_{\sigma^n}     \ \ \text{and}   \ \    \rho_{n,1}=g_{\sigma^n}\circ (\,\hat i\times 1),     \]
i.e., $g_{\sigma^n}$ makes the following diagram commutative
\[
\begin{array}{cccccc}
  \hat I^{n,1} \times I^q&     \overset{ {\rho}_{n,1}}{\longrightarrow}      &  E\\
        \hspace{-0.3in} \hat i\times 1 \downarrow     &               {\nearrow}    &  \hspace{0.1in}    \downarrow \pi \\
    I^n\times I^q &  \overset{f_{\sigma^n}}{\longrightarrow}  &    Y.
\end{array}
\]

\end{proposition}
\begin{proof}
The proof is by induction on the dimension $n.$  Given $2\leq k\leq n,$
let $\hat I^k\subset I^k$ the union of the $(k-1)$-faces $d^{\epsilon}_i(I^k)$ of $I^k$ except the face $d^0_1(I^k),$
and let  $\hat i: \hat I^k\hookrightarrow I^k.$ For $k=2$ use the homeomorphism
 of pairs $(I\times I, 1\times I)\approx (I^2, \hat I^2)$
 to lift $f_{\sigma^2}$ to a map $g_{\sigma^2}: I^2\times I^q\rightarrow E$ that extends $\rho_{2,1}.$
In particular, $g_{\sigma^2}|_{d^0_1(I^2)\times I^q}\subset \pi^{-1}(\min \sigma^2),$ and, consequently, the map
 \[\omega_{\sigma^2}:d^0_1(I^{2})\times I^q\rightarrow  G\]
is also defined. Suppose $ g_{\sigma^k} :I^k\times I^q\rightarrow E$
and  $ \omega_{\sigma^k} :d^0_1(I^k)\times I^q\rightarrow G$
for  $I^k\subset I^n$ and $2\leq k<n$     has been constructed satisfying the hypotheses of the proposition.
For $k=n$
define the map $\rho_{n}: \hat I^n\times I^q\rightarrow E  $    by
\[ \rho_{n} |_{_{I^{n-1}\times I^q}}=\!
\left\{\!\!
  \begin{array}{llll}
    g_{\sigma^{n-1}}, &        I^{n-1}=d^1_i(I^n),  & 1\leq i\leq n      ,\vspace{5mm}\\
   g_{\sigma^{i-1}} \! \cdot \!\omega_{\sigma^{n-i+1}},                 &
                      I^{n-1}\!= d^0_i(I^n)= \! I^{i-1}\!\!\times\! I^{n-i},          &  1< i\leq n  ,
  \end{array}
\right.
 \]
where $\sigma^{i-1}\times \sigma^{n-i+1} $ is  a component of the AW decomposition of $\sigma^n$ and $g_{\sigma^{i-1}}  \cdot \omega_{\sigma^{n-i+1}}$ is the map $I^{i-1} \times (I^{n-i} \times I^q) \to E$ defined by $g_{\sigma^{i-1}} \cdot \omega_{\sigma^{n-i+1}}(t_1,...,t_{n+q-1})= g_{\sigma^{i-1}}(t_1,...,t_{i-1}) \cdot \omega_{\sigma^{n-i+1}}(t_i,...,t_{n+q-1})$.
  Then  the equivariance and $d^0$-equivariance of the maps $g_{\sigma^k}$ guarantee that  $\rho_n $
is well defined and the diagram
\[
\begin{array}{cccccc}
  \hat I^n\times I^q &  &     \overset{ {\rho}_n}{\longrightarrow}   &  & &  E\\
      \hspace{-0.3in}   \hat i\times 1 \downarrow        &    &          &     &                    &\!\!\!\!\! \pi \downarrow \\
    I^n\times I^q &  \overset{pr_n}{\longrightarrow} I^n &  \overset{\psi}{\longrightarrow}  &    \Delta^n &  \overset{\sigma^n}{\longrightarrow} &    Y.
\end{array}
\]
commutes.
Use the homeomorphism  of pairs $( I \times I^{n-1}, 1\times I^{n-1})\approx (I^n, \hat I^n)$
 to lift $f_n$ to a map $g_{\sigma^n}: I^n\times I^q\rightarrow E$ which extends $\rho_n.$
In particular, $g_{\sigma^n}|_{d^0_1(I^n)\times I^q}\subset \pi^{-1}(\min \sigma^n),$ and, consequently, the map
 \[\omega_{\sigma^n}:d^0_1(I^{n})\times I^q\rightarrow  G\]
is also constructed.
\end{proof}

Our main statement is the following

\begin{theorem}\label{loopmodel}
Let $Y=|X|$ be the geometric realization of a path connected simplicial set $X.$
Let $ \Omega Y \overset{i}{\rightarrow} PY\overset{\pi} \rightarrow Y$ be  the (Moore) path fibration on $Y.$ Then  there is a  commutative diagram
\begin{equation}\label{diagram}
\begin{array}{cccccc}
   |\widehat{\mathbf{\Omega}} X| & \overset{\omega}{\longrightarrow}                  &                 \Omega Y \\
        \iota  \downarrow            &                  &      \hspace{-0.1in} \iota   \downarrow \\
      |\widehat{\mathbf{P}}   X|  &  \overset{p}{\longrightarrow} &                     PY                    \\
     \xi \downarrow    &                                    &       \hspace{-0.1in}\pi     \downarrow           \\
 \hspace{0.1in}|X| & \overset{Id}{\longrightarrow}                    &                     Y
\end{array}
\end{equation}
in which
  $\omega$ is a monoidal map and homotopy equivalence.
\end{theorem}
\begin{proof}
The constructions of the maps  $p$ and $\omega$ are in fact  simultaneous by induction on the dimension of  simplices in $X$.
Given
 $|(x,y)|\in  | X_0\times \widehat{\mathbf{\Omega}}'_0 X | \to |\widehat{\mathbf{P}}X |$
with $y=\bar \sigma_1\cdots\bar \sigma_k ,$
  define $p(x, y)\in PY$ as the edge-path $\sigma_1\cup \cdots\cup \sigma_k$ with the initial vertex $x$ and the final vertex $x_0.$
Let  $ (v,y)\in |(x, y)|\subseteq   |X_1\times \widehat{\mathbf{\Omega}}'_0 X| \to |\widehat{\mathbf{P}}X|$ for $(x, y)\in X_1
\times \widehat{\mathbf{\Omega}}'_0 X$ with $x\in X_1$ non-degenerate.
Define $p(v,y)=  [v,\max x]\cdot \omega(y)\in PY,$ where $[v,\max x]\subset |x|$ is the interval.
Assume $\omega$ is defined on $\widehat{\mathbf{\Omega}}_r X$ for $0\leq r\leq q.$
Let  $ |(\sigma^n, y)|\subseteq  |X_n\times \widehat{\mathbf{\Omega}}'_q X|\to |\widehat{\mathbf{P}}X|$ for
$(\sigma^n, y)\in  X_n\times \widehat{\mathbf{\Omega}}'_q X \to \widehat{\mathbf{P}}X$ with $\sigma^n$ non-degenerate.
Then the hypotheses of Proposition  \ref{lift} are satisfied   for $\sigma^n,$ and let $ g_{\sigma^n} :I^n\times I^q\rightarrow E $
be the resulting map.
Define $p|_{|(\sigma^n, y)|}$ to be induced by  $g_{\sigma^n},$ and then $\omega$ is extended on
$\widehat{\mathbf{\Omega}}_r X$ for $0\leq r\leq q+n-1$
by  means of the maps $\omega_{\sigma^n}$ for all non-degenerate $\sigma^n\in X_n.$
Thus the maps $p$ and $\omega$ are constructed such that $ \iota \circ \omega=p\circ \iota.$ Finally apply Corollary \ref{coro}
 to finish the proof.
\end{proof}

\begin{remark} \normalfont If $X$ is a path connected Kan complex the restriction of $\omega$ to   $|\mathbf{\Omega}X| \to \Omega Y$ is also
 a homotopy equivalence. The same proof as the one above works for this case as well since if $X$ is a Kan complex then $\pi_0(|\mathbf{\Omega}X|)$ is also a group because each $1$-simplex in $X_1$ has an inverse up to homotopy. In particular, let $Y$ be a connected topological space with a base point $y \in Y$ and let $\operatorname{Sing}(Y,y)$ be the Kan complex consisting of singular simplices whose vertices are mapped to $y$. Then $|\mathbf{\Omega} \operatorname{Sing}(Y,y)|$ is homotopy equivalent to $\Omega Y$.
\end{remark}

\section{Algebraic models for the based loop space and the hat-cobar construction.}
\subsection{The hat-cobar construction} We fix a ground commutative ring $\Bbbk$  with unit $1_{\Bbbk}$. All modules are assumed to be  over  $\Bbbk.$
Since the face maps of
$\widehat{\mathbf{\Omega}} X$ satisfy the standard cubical set relations we can  form
the  chain complex
$(C'_\ast(\widehat{\mathbf{\Omega}} X),d )$
 with the differential
$d=\sum_{i=1}^n(-1)^i (d^1_i - d^0_i): C'_n(\widehat{\mathbf{\Omega}} X)\rightarrow C'_{n-1}(\widehat{\mathbf{\Omega}} X),$ but we refer to the chain complex
\[C_\ast(\widehat{\mathbf{\Omega}} X)=C'_\ast(\widehat{\mathbf{\Omega}} X)/C'_\ast(D(e)),\]
where $D(e)\subset \widehat{\mathbf{\Omega}} X$ denotes the set  of degeneracies arising from the unit $e,$
  as the chain complex of the necklical set $\widehat{\mathbf{\Omega}} X.$
 Moreover, $C_\ast(\widehat{\mathbf{\Omega}} X)$ is a dg algebra (since $\widehat{\mathbf{\Omega}} X$ is a monoidal necklical set) which calculates the singular homology of $|\widehat{\mathbf{\Omega}} X|$.
We also can form a quasi-isomorphic  dg algebra  $C^{N}_\ast(\widehat{\mathbf{\Omega}} X)=
C'_\ast(\widehat{\mathbf{\Omega}} X)/C'_\ast(D(\widehat{\mathbf{\Omega}} X)),$
where this time  $D(\widehat{\mathbf{\Omega}} X)\subset \widehat{\mathbf{\Omega}} X$ denotes the set of all degenerate elements.

For a  module $M,$ let $T(M)$ be the tensor algebra of $M$, i.e. $
T(M)=\oplus _{i=0}^{\infty }M^{\otimes i}$. We denote by $s^{-1}M$ the desuspension of
$M$, i.e. $(s^{-1}M)_{i}=M_{i+1}$. An element
$s^{-1}a_{1}\otimes ...\otimes s^{-1}a_{n}\in (s^{-1}M)^{\otimes n}$ is denoted by
$[a_{1}|...|a_{n}]$.

Given a   simplicial set $X,$ we may consider the dg coalgebra $(C_\ast Z(X), d_C,\Delta)$ where $d_C$ is the usual boundary map on the chains $C_\ast Z(X)$  and $\Delta$ is the Alexander-Whitney diagonal map. Let $C_{\ast > 0} (X_0)$ denote the sub dg coalgebra of $C_\ast Z(X)$ generated by simplices of positive degree on the subsimplicial set of $Z(X)$ generated by the vertices $X_0=Z(X)_0$, so all generators of $C_{\ast > 0} (X_0)$ are degenerate simplices having degenerate faces.
We may truncate $d_C$ and $\Delta$ to obtain a new dg coalgebra $A:=(A_\ast(X), d_A,\Delta')$  where
 \[ A_\ast(X)=C_\ast Z(X) / C_{\ast > 0} (X_0), \,\,
d_A=d_C-\partial_0-(-1)^{n}\partial_n: A_n\rightarrow  A_{n-1},\]
and $\Delta'$ is $\Delta$ without the primitive  term.
Let $(\Omega A,d_{\Omega})$ be the cobar construction of $A,$ i.e.,
$\Omega A=T(s^{-1} \bar A_*(X))$ is the tensor algebra on the desuspension of $\bar A_*(X) := A_{*>0}(X)$
 with the differential
$d_\Omega=d_{1}+d_{2}$ defined for $\bar{a}\in\bar{A}$ by
\begin{equation*}
d_{1}[\,\bar{a}\,]=-\left[\,\overline{d_{A}(a)}\, \right]
\end{equation*}
and
\begin{equation*}
d_{2}[\, \bar{a}\, ]=
\sum(-1)^{|a^{\prime}|}\left[\, \bar{a^{\prime}}\mid\bar{a^{\prime\prime}}\,\right],\ \ \
\text{for}\ \ \
{\Delta^{\prime}}(a)=\sum a^{\prime}\otimes a^{\prime\prime},
\end{equation*}
extended as a derivation.
Define a submodule
$\Omega'_n A\subset \Omega _nA$
as generated by monomials $[\bar a_1|\cdots| \bar a_k]\in \Omega_n A,\,k\geq 1,  $   where each $a_i$ is a simplex in $Z(X)$ representing a generator of $A$ such that
 $\min a_1=\max a_k= x_0 $ and
 $\max a_i=\min a_{i+1}$
for all $i;$ in particular, $\Omega' A= \Omega A$ when $X_0=\{x_0\}$.
 Define the \emph{hat-cobar construction}  $\widehat{\Omega}C_\ast(X)$ of the dg coalgebra $C_\ast(X)$ as
\[  \widehat{\Omega}C_\ast(X)=\Omega' A /\sim ,  \]
where $\sim$ is generated by
\[
[\bar a_1|...|\bar a_{i-1}|\bar a_i|\bar a_{i+1}|\bar a_{i+2}|...|\bar a_k ]\sim [\bar a_1|...|\bar a_{i-1}|\bar a_{i+2}|...|\bar a_k]
\ \  \text{whenever}\ \  a_{i+1}=a_i^{op};
\]
in particular, $[\bar a_i|\bar a_{i+1}]\sim 1_{\Bbbk}.$

For a $1$-reduced $X$ (e.g., $X=\operatorname{Sing}^1(Y,y)$ the simplicial set consisting of all singular simplices in a topological space $Y$ which collapse edges to a fixed point $y \in Y$), the hat-cobar construction $\widehat{\Omega} C_\ast(X)$ coincides with the Adams' cobar construction  $\Omega C_\ast(X)$ of  the dg coalgebra $C_\ast(X)$. We have an obvious

\begin{theorem}\label{hat-cobar} For a simplicial set $X$
the dg algebra  $C_\ast(\widehat{\mathbf{\Omega}} X)$    coincides with the hat-cobar construction $\widehat{\Omega}C_\ast(X).$
\end{theorem}
Note that the similar theorem is true for $C^N_\ast(\widehat{\mathbf{\Omega}} X)$ and the \emph{normalized}  hat-cobar construction
$\widehat{\Omega}^NC_\ast(X)$ obtained from the definition of the hat-cobar construction by replacing $X_0$ by $X_0\cup D(X).$

From Theorems \ref{loopmodel} and \ref{hat-cobar} and the homotopy invariance of the based loop space we have
\begin{proposition}
If $f_\ast:C_\ast(X)\rightarrow C_\ast(X')$ is induced by a weak equivalence $f: X\rightarrow X',$
then $\widehat{\Omega}f_\ast:\widehat{\Omega}C_\ast(X) \rightarrow \widehat{\Omega}C_\ast(X') $ is a quasi-isomorphism.
\end{proposition}

\begin{remark}
\normalfont It follows from Remark 3 that for a path connected topological space $Y$ the dg algebras $C_*(\mathbf{\Omega} \operatorname{Sing}(Y,y) )$ and $C^N_*(\mathbf{\Omega} \operatorname{Sing}(Y,y) )$ are quasi-isomorphic to the dg algebra of singular chains on $\Omega Y$. Moreover, $C_*(\mathbf{\Omega} \operatorname{Sing}(Y,y) )$ and $C^N_*(\mathbf{\Omega} \operatorname{Sing}(Y,y) )$ are isomorphic to the cobar construction on the dg coalgebras $C_*(\operatorname{Sing}(Y,y))/C_{*>0}(y)$ and $C^N_*(\operatorname{Sing}(Y,y) )$, respectively. Thus Adams' classical cobar construction provides a dg algebra model for the based loop space of a space $Y$ - even when $Y$ is non-simply connected- when applied to a Kan complex model $Y$ such as $\operatorname{Sing}(Y,y)$. This fact was also recently stated and proved by the first author and M. Zeinalian in \cite{Rivera- Zeinalian} using results from J. Lurie's theory of $\infty$-categories.
\end{remark}

\subsection{The hat-cobar construction of a simplicial  set X with a single vertex}
Let  $(A,d_A,\Delta)$ be a dg coalgebra such that the module of cycles $Z_1(A)\subset A_1$ is free  with basis $\mathcal{Z}_1.$
Let $G_1$ be the free group generated by $\mathcal{Z}_1$ and let $\Bbbk[G_1]$ be the group ring.
Define a graded module $A[1]$ as
  $A[1]_0=A_0$, $A[1]_1=\Bbbk[G_1]$ and $A[1]_i=A_i$ for $i\geq 2.$  Then $A\subset A[1]$ extends  to the dg coalgebra
 $(A[1], d, \Delta)$ (with $d(A[1]_1)=0$).

Define the hat-cobar construction  $(\widehat{\Omega} A,d_{\hat \Omega})$ of $A$  as
the standard cobar construction
${\Omega} A[1]$ of $A[1]$ modulo the relations $[\,\bar 1_{G_1}]= 1_{\Bbbk}$ and
\begin{multline*}
[\bar a_1|...|\bar a_{i-1}|\bar a_i|\bar a_{i+1}|\bar a_{i+2}|...|\bar a_k ]= [\bar a_1|...|\bar a_{i-1}|\overline {a_i a_{i+1}}\,|\bar a_{i+2}|...|\bar a_k]
\ \  \text{whenever}\\  a_{i},a_{i+1}\in G_1.
\end{multline*}

Then given a   simplicial set $X$ with $X_0$ to be a singleton  the hat-cobar construction of $X$ is
$\widehat{\Omega} A$ for the dg coalgebra $A=(C_\ast(X), d_A,\Delta)$ where $d_A=-\partial_1+\cdots +(-1)^{n-1}\partial_{n-1}: A_n\rightarrow  A_{n-1}$. In this case we recover the \textit{extended cobar construction} introduced in \cite{Hess- Tonks}.

\bibliographystyle{plain}

 \Addresses

\end{document}